\documentclass[12pt,reqno]{amsart}

\usepackage{amsfonts,amssymb,amstext,amsmath,amscd,latexsym}
\usepackage[dvips]{graphicx}

\theoremstyle{plain}
\newtheorem{theorem}{Theorem}[section]

\newtheorem{corollary}{Corollary}[theorem]

\theoremstyle{definition}
\newtheorem{definition}{Definition} 

\theoremstyle{remark}

\newtheorem{remark}{Remark}[section]

\title[Isometry group and geodesics of the Wagner lift]{Isometry group and geodesics of the Wagner lift of a riemannian metric on two-dimensional manifold }

\author{Jos\'e Ricardo Arteaga B.}
\address{Universidad de los Andes,
Departamento de Matem\'aticas,  Cra. 2 ESTE 18A-10, Bogot\'a, Colombia}
\email{jarteaga@uniandes.edu.co}
\author{Mikhail Malakhaltsev}
\address{Universidad de los Andes,
Departamento de Matem\'aticas,  Cra. 2 ESTE 18A-10, Bogot\'a, Colombia;
Kazan State University,
Department of Mathematics, Kremlevskaya, 8, Kazan, Russia}
\email{mikarm@gmail.com}

\subjclass{53B20, 53B21, 53C17}

\keywords{Lift of a metric tensor, fiber bundles, Wagner lift, sub-riemannian metrics,
metric with singularities, isometry group, geodesics, geodesic simulation}

\def\MMR#1{\relax}

\begin{document}
\maketitle

\begin{abstract}
In this paper we construct a functor from the category of two-dimensional Riemannian manifolds to the category of three-dimensional manifolds with generalized metric tensors.  For each two-dimensional oriented Riemannian manifold $(M,g)$ we construct a metric tensor $\hat g$ (in general, with singularities) on the total space $SO(M,g)$ of the principal bundle of the positively oriented orthonormal frames on $M$. We call the metric $\hat g$ the Wagner lift of $g$.  We study the relation between the isometry groups of $(M,g)$ and $(SO(M,g),\hat g)$. We prove that the projections of the geodesics of $(SO(M,g),\hat g)$ onto $M$ are the curves which satisfy the equation  
\begin{equation*}
\nabla_{\frac{d\gamma}{dt}}\frac{d\gamma}{dt} = C K J (\dot\gamma) - C^2 K grad K, 
\end{equation*} 
where $K$ is the curvature of $(M,g)$, $J$ is the operator of the complex structure associated with $g$, and $C$ is a constant. We find the properties of the solutions of this equation, in particular, for the case when $(M,g)$ is a surface of revolution.  
\end{abstract}

\section{Introduction}

A sub-Riemannian manifold is a triple $(Q,\mathcal{H}, g_{\mathcal{H}})$, where $Q$ is a manifold and $\mathcal{H}$ is a distribution on $Q$ endowed by a metric tensor $g_{\mathcal{H}}$ \cite{Montgomery}. 
One of the first geometers who systematically studied sub-Riemannian manifolds was V.V. Wagner \cite{Wagner}. 
For a sub-Riemannian manifold $(Q,\mathcal{H},g_{\mathcal{H}})$  he defined a connection $\nabla^{\mathcal{H}}$ on $Q$ which determines a parallel translation of vectors belonging to the distribution $\mathcal{H}$ along the curves tangent to $\mathcal{H}$ (a ``partial'' connection) compatible with the metric $g_{\mathcal{H}}$  and subordinate to a ``torsion-free'' condition. 
V.\,Wagner also constructed a series of tensors which all together play the role of the curvature tensor of a linear connection. This means that  the parallel translation with respect to the connection $\nabla^{\mathcal{H}}$ does not depend on the path if and only if all these curvature tensors vanish. 
 One of the main steps of V.\,Wagner's construction is an extension of the metric $g_{\mathcal{H}}$ given on $\mathcal{H}$ to the first prolongation $\mathcal{H}^1$ of $\mathcal{H}$, where $\mathcal{H}^1$ is the distribution generated by the Lie brackets of vector fields tangent to $\mathcal{H}$. 
This extension is of interest in itself and, as we will show, gives possibility to construct a natural lift of a Riemannian metric $g$ on a two-dimensional manifold $M$ to its orthonormal frame bundle (in fact, to construct a functor from the category of two-dimensional Riemannian manifolds to the category of three-dimensional Riemannian manifolds). It turns out that this lift is closely related to many classical constructions in the geometry of fiber bundles (see \cite{BN}, \cite{Shapukov1}, \cite{Shapukov2}), and to special sub-Riemannian manifolds considered in (\cite{Montgomery}, Ch. 11).    

Let us describe V.\,Wagner's extension of metric in the case we will use: $\dim Q = 3$ and $\dim \mathcal{H}=2$.
Assume that on a three-dimensional manifold $Q$ there is given a two-dimensional distribution $\mathcal{H}$ with a metric tensor $g_{\mathcal{H}}$. 
In addition, assume that the complementary distribution $\mathcal{N}$ is given on $Q$ 
(i.\,e. $TQ = \mathcal{H}\oplus \mathcal{N}$). 
V.\,Wagner introduces the nonholonomity  tensor $N : \mathcal{H} \times \mathcal{H} \to \mathcal{N}$, $N(X,Y) = pr_\mathcal{N}([X,Y])$, where $pr_\mathcal{N}$ is the projector onto $\mathcal{N}$ with respect to $\mathcal{H}$. 
The distribution $\mathcal{H}$ is said to be completely nonholonomic if the map $N$ is surjective, or equivalently the first prolongation $\mathcal{H}^1$ coincides with $TQ$. 
For a completely non-holonomic $\mathcal{H}$, one can extend the metric $g_{\mathcal{H}}$ onto $\mathcal{H}^1 = TQ$ in the following way. 
Take an orthonormal frame $E_1, E_2$ of $\mathcal{H}$ with respect to $g_{\mathcal{H}}$, and set $E_3 = N(E_1,E_2)$. Then, the extended metric $g^1_{\mathcal{H}}$ on $TQ$ is uniquely determined by the condition that $\{E_1, E_2,E_3\}$ is an orthonormal frame of $TQ$ with respect to $g^1_{\mathcal{H}}$. 

In \cite{Arteaga-Malakhaltsev} we  applied V.\,Wagner's metric extension to the following specific case. 
Let us consider a $2$-dimensional oriented Riemannian manifold $(M,g)$, and denote by $SO(M,g) \to M$ the $SO(2)$-principal bundle  of positively oriented orthonormal frames of $(M,g)$. 
The metric $g$ determines the Levi-Civita connection on $M$ which, in turn, determines the horizontal distribution $\mathcal{H}$ on $SO(M,g)$. 
We take the lift $g_\mathcal{H}$ of $g$ to $\mathcal{H}$ and thus obtain a sub-Riemannian manifold $(SO(M,g),\mathcal{H},g_\mathcal{H})$. 
For the complementary distribution $\mathcal{N}$ of $\mathcal{H}$ we take the vertical distribution (the distribution tangent to fibers of the bundle $SO(M,g) \to M$).  
And thus, using V.\,Wagner's construction, we get a metric $\hat g$ on $SO(M,g)$ as the extension of $g_{\mathcal{H}}$, and call $\hat g$ the \textit{Wagner lift} of the metric $g$ on $M$. 
In \cite{Arteaga-Malakhaltsev} we calculated the curvature tensor of $(SO(M,g),\hat g)$, found the differential equations for the projections of geodesics of  $(SO(M,g),\hat g)$ onto $M$ (see \eqref{eq:2_28} of the present paper), and described the Wagner lift for the metrics $g$ of constant curvature $K \ne 0$. 
It is worth noting that V.\,Wagner's metric extension can be applied only to completely nonholonomic distributions, and in our case the distribution $\mathcal{H}$ is completely nonholonomic if and only if the curvature $K$ of $(M,g)$ does not vanish. Therefore, in \cite{Arteaga-Malakhaltsev} we assumed that the curvature $K \ne 0$. 

In the present paper we consider the Wagner lift $\hat g$ of a metric $g$ on a two-dimensional oriented manifold $M$ under the assumption that the curvature of $g$ can vanish along a one-dimensional submanifold $\Sigma \subset M$. 
In this case, the metric tensor $\hat g$ is a \emph{metric tensor with singularities} on $SO(M,g)$. This means that $\hat{g}$ is not defined at the points of the surface $\hat{\Sigma}=\pi^{-1}(\Sigma) \subset SO(M,g)$, however there is  a tensor $\hat g^*$ which is defined everywhere on $SO(M,g)$ and $\hat g^*$ is dual to $\hat g$ wherever $\hat g$ is defined (certainly $\hat{g}^*$ is degenerate on $\hat{\Sigma}$).  

The paper is organized as follows. In subsection~\ref{subsec:2_4} we present the precise definition of the Wagner lift, subsection~\ref{subsec:2_6}  contains formulas expressing the basic geometrical objects of the metric $\hat{g}$ on $SO(M,g)$ in terms of the metric $g$ on $(M,g)$ (without proof, for the proofs see \cite{Arteaga-Malakhaltsev}). 
In section~\ref{sec:3} we find relation between the isometry group of $(M,g)$ and the isometry group of $(SO(M,g),\hat g)$. 
In section~\ref{sec:4} we study geodesics of $(SO(M,g), \hat g)$ and their projections onto $M$. Namely, in subsection~\ref{subsec:4-1} we investigate in details the correspondence between the geodesics of $(SO(M,g), \hat g)$ and their projections (Theorem~\ref{thm:3_3}) and describe  the behavior of the geodesics near the surface $\hat{\Sigma}$ (Theorem~\ref{prop:4-1}). 
In subsection~\ref{subsec:4-2} we describe geodesics of  $(SO(M,g), \hat g)$ and their projections onto $M$ from the point of view of the theory of geodesic modeling developed in \cite{ShapiroIgoshinYakovlev}, \cite{Yakovlev}, \cite{Igoshin1}, and present a Lagrangian for the projections of geodesics which are solutions to the differential equation system~\eqref{eq:2_28}.
In subsection~\ref{subsec:4-3} we study solutions of   \eqref{eq:2_28} in case $(M,g)$ is a surface of revolution, find two first integrals of the system, and formulate some properties of solutions. Several pictures of solution behavior on the torus are included (to do it we used SAGE www.sagemath.org).

Finally we note that  the Wagner lift can be constructed just in the same way without assumption of orientability of $M$, in this case we will have a metric $\hat g$ (with singularities) on the orthonormal bundle frame and similar results hold true.

\section{Wagner lift of a metric $g$ on $M$ to $\hat{g}$ on $SO(M,g)$}

\subsection{Preliminaries} 
Let $M$ be a $2$-dimensional oriented manifold and $g$ be a metric tensor on $M$.
We denote by $\pi : SO(M,g) \to M$ the principal $SO(2)$-bundle of positively oriented orthonormal frames of $(M,g)$. In what follows, for the elements of $SO(2)$ we use the notation 
\begin{equation}
r(\varphi) = 
\left(
\begin{array}{cc}
\cos \varphi & -\sin \varphi
\\
\sin \varphi & \cos \varphi
\end{array}
\right). \notag
\end{equation}

The Levi-Civita connection $\nabla$ of $g$ determines an infinitesimal connection $\mathcal{H}$ in the bundle $SO(M,g) \to M$, that is an $SO(2)$-invariant distribution on the total space $SO(M,g)$ which is complementary to the fibers of the bundle. In what follows we use the results and terminology of the theory of connections in fiber bundles (see, e.\,g., (\cite{KN}, Ch.I--III)). 

Let $U \subset M$ be an open set such that there exists a section  $s(x) = \{e_1(x),e_2(x)\}$, $x \in U$, of the bundle $\pi : SO(M,g) \to M$.  The section $s$ determines  a trivialization 
\begin{equation}
\psi: U \times SO(2) \to \pi^{-1}(U), \quad (x, r(\varphi)) \mapsto s(x) r(\varphi).
\label{eq:1}
\end{equation}
In terms of the trivialization \eqref{eq:1} the connection form $\omega$ on $SO(M,g)$ determined by the Levi-Civita connection $\nabla$ is written as follows:
\begin{equation}
\omega_{(x,\varphi)} = r(\varphi)^{-1}
 \left(\frac{dr(\varphi)}{d\varphi}d\varphi + \Gamma_a(x) r(\varphi)\hat\theta^a\right).
\label{eq:3_4}
\end{equation}
Here $\hat\theta^a = \pi^*\theta^a$, $a=1,2$, where $\{\theta^a\}$ is the coframe field on $U$ dual to $\{e_a\}$; $\Gamma_a = \Gamma_a(x)$ is a 1-form on $U$ with values in the Lie algebra $\mathfrak{o}(2)$:
\begin{equation*}
\Gamma_{a}=
\left(
\begin{array}{cc}
0 & \Gamma_{a2}^{\hphantom{1} 1}\MMR{\Gamma_{a2}^{\hphantom{1} 1}}\\
\Gamma_{a1}^{\hphantom{1} 2} \MMR{\Gamma_{a1}^{\hphantom{1} 2}}& 0
\end{array}
\right), \notag
\label{eq:1000}
\end{equation*}
where $\Gamma_{ab}^{\hphantom{1}c}$ are the coefficients of $\nabla$ with respect to the orthonormal frame field $\{e_a\}$. The matrix $\Gamma_a = ||\Gamma_{ab}^{\hphantom{1}c}||$ is antisymmetric  
($\Gamma_{ab}^{\hphantom{1}c}=-\Gamma_{ac}^{\hphantom{1}b})$ because $\{e_a\}$ is an orthonormal frame field and $\nabla$ is the Levi-Civita connection. 

As $SO(2)$ is a commutative group we have that
$r(\varphi)^{-1} \Gamma_a r(\varphi)= \Gamma_a$. 
In addition,  $\frac{dr}{d\varphi} = r(\varphi+\pi/2)$, so 
\begin{equation}
r^{-1}(\varphi) \frac{dr}{d\varphi} = r(\pi/2) = 
J = 
\left(
\begin{array}{cc}
0 & -1
\\
1 & 0
\end{array}
\right). \notag
\end{equation}

Now we can rewrite \eqref{eq:3_4} as follows: 
\begin{equation}
\omega = J d\varphi + \Gamma_a \hat\theta^a, \notag
\end{equation}
or, in the matrix form, as  
\begin{equation}
\omega = 
\left(
\begin{array}{cc}
0 & -\alpha
\\
\alpha & 0
\end{array}
\right)
= 
\left(
\begin{array}{cc}
0 & -d\varphi 
\\
d\varphi & 0
\end{array}
\right)
+
\left(
\begin{array}{cc}
0 & \Gamma_{12}^{\hphantom{1}1}\hat\theta^1
\\
\Gamma_{11}^{\hphantom{1}2} \hat\theta^1 & 0
\end{array}
\right)
+
\left(
\begin{array}{cc}
0 & \Gamma_{22}^{\hphantom{1}1}\hat\theta^2
\\
\Gamma_{21}^{\hphantom{1}2}\hat\theta^2 & 0
\end{array}
\right), \notag
\end{equation}
hence  
\begin{equation*}
\alpha = d\varphi + \Gamma_{11}^{\hphantom{1}2} \hat\theta^{1} + \Gamma_{21}^{\hphantom{1}2}\hat\theta^{2} = d\varphi - \Gamma_{12}^{\hphantom{1}1} \hat\theta^{1} - \Gamma_{22}^{\hphantom{1}1}\hat\theta^{2}. 
\label{eq:3_5}
\end{equation*}
This implies that the horizontal lifts $E^h_1$, $E^h_2$ of the vector fields  $e_{1}$, $e_{2}$ can be written as follows:
\begin{equation}
E^h_1(x,\varphi) = e_1 - \Gamma_{11}^{\hphantom{1}2} \partial_\varphi,\quad 
E^h_2(x,\varphi) = e_2 - \Gamma_{21}^{\hphantom{1}2} \partial_\varphi.
\label{eq:3_6}
\end{equation}
Note also that, for a matrix
$A = \left(
\begin{array}{cc}
0 & -m
\\
m & 0
\end{array}
\right) \in \mathfrak{o}(2)$, $m\in \mathbf{R}$, 
the fundamental vector field $\sigma( A ) = m \partial_\varphi$.

\subsection{Expressions via structure functions}
For an orthonormal frame field  $\{e_1,e_2\}$ one can define the \emph{structure functions}  $c^k_{ij}$:
\begin{equation}
[e_i,e_j] = c^{\hphantom{1}k}_{ij} e_k. \notag
\end{equation}
From the well-known formula (see, e.\,g., \cite{KN}, p.~160): 
\begin{eqnarray*}\label{eq:3_7}
g(\nabla_X Y,Z) &=& \frac{1}{2}(X g(Y,Z) + Y g(X,Z) - Z g(X,Y)\nonumber\\ 
						&+& g([X,Y],Z) + g([Z,X],Y) - g([Y,Z],X)).
\end{eqnarray*}
we get that that the coefficients of $\nabla$ with respect to the orthonormal frame $\{e_i\}$ can be written as follows: 
\begin{equation*}
\Gamma^{\hphantom{1}k}_{ij} = \frac{1}{2}(c^{\hphantom{1}k}_{ij} + c^{\hphantom{1}j}_{ki} + c^{\hphantom{1}i}_{kj}).
\label{eq:3_8}
\end{equation*}
For the curvature tensor 
\begin{equation*}\label{eq:3_9}
R(X,Y)Z = \nabla_{X}\nabla_{Y}Z - \nabla_{Y}\nabla_{X}Z - \nabla_{[X,Y]}Z
\end{equation*}
of the connection $\nabla$ we have
\begin{equation}\label{eq:3_11}
R_{ijk}^l = e_i\Gamma^{\hphantom{1}l}_{jk} - e_j\Gamma^{\hphantom{1}l}_{ik} + \Gamma^{\hphantom{1}l}_{is}\Gamma^{\hphantom{1}s}_{jk} 
- \Gamma^{\hphantom{1}l}_{js}\Gamma^{\hphantom{1}s}_{ik} - c^{\hphantom{1}s}_{ij} \Gamma^{\hphantom{1}l}_{sk},
\end{equation}
this gives us the components $R_{ijk}^l$  in terms of $c^{\hphantom{1}k}_{ij}$.

Note also that   \eqref{eq:3_6} can be rewritten in terms of the structure functions as follows:
\begin{equation}\label{eq:2_13}
E^h_1(x,\varphi) = e_1 + c^{\hphantom{1}1}_{12}(x) \partial_\varphi,
\quad
E^h_2(x,\varphi) = e_2 + c^{\hphantom{1}2}_{12}(x) \partial_\varphi.
\end{equation} 

Let us apply V.\,Wagner's construction of nonholohomity tensor to the distribution $\mathcal{H}$. For the complementary distribution $\mathcal{N}$ we take the distribution $V$ of vertical subspaces of the bundle $SO(M,g) \to M$, i.\,e. the distribution of the tangent spaces to the fibers of the bundle. 
Then the \emph{nonholonomity tensor} $ N : \Lambda^{2}(\mathcal{H}) \to V$, where $\Lambda^{2}(\mathcal{H})$ is the space of bivectors on the distribution $\mathcal{H}$, is defined as follows: 
\begin{equation}\label{eq:3_13}
N(X,Y)_p = pr_V ([\tilde X, \tilde Y]_p), \quad p \in SO(M,g),
\end{equation}
where $pr_V$ is the projection onto $V$ parallel to $\mathcal{H}$, and  $\tilde X$, $\tilde Y$ are vector fields such that $\tilde X(p) = X$ and $\tilde Y(p)=Y$. 

Let $e=\left\{e_{1}, e_{2}\right\}$ be an orthonormal frame field on the base $M$, and $E_{1}^{h}$, $E_{2}^{h}$ be the horizontal lifts of $e_1$, $e_2$, respectively. Then 
\begin{equation}
\label{eq:3_16}
N(E^h_1(x,\varphi),E^h_2(x,\varphi)) =  K(x) \partial_\varphi,
\end{equation} 
where $K(x)$ is the curvature of $(M,g)$ at $x \in M$ (see \cite{Arteaga-Malakhaltsev}, Prop.~3).

\subsection{Definition of the Wagner lift of metric}
\label{subsec:2_4}
Recall that we consider a two-dimensional oriented Riemannian manifold $(M,g)$. 
Denote by $\Sigma$ the set of points $x \in M$ such that $K(x)=0$, and let $\hat \Sigma = \pi^{-1}(\Sigma)$. Note that $M \setminus \Sigma$ is an open subset of $M$, and $SO(M,g) \setminus \hat\Sigma$ is an open subset of $SO(M,g)$. 

At each point $p \in SO(M,g)$ such that $\pi(p) \in M \setminus \Sigma$, we will construct  a metric tensor $\hat g$ on $T_p SO(M,g)$.  We have $T_p SO(M,g)=\mathcal{H}_p \oplus V_p$, hence it is sufficient to define $\hat g$ on the horizontal and vertical vectors. Note also that the vector field $\partial_\varphi$ is globally defined and spans the vertical distribution $V$, therefore the value $\hat g(\partial_\varphi,\partial_\varphi)$ determines the values of $\hat g(X,Y)$ for arbitrary vertical vectors $X$, $Y$. 
\begin{definition}\label{def:4_1}
For each point $p \in SO(M,g)$ such that $K(\pi(p))\ne 0$, 
\begin{eqnarray}
\hat g(X,Y)_p &=& g(d\pi(X),d\pi(Y)), \quad \forall X,Y \in \mathcal{H}_p,  
\label{eq:4_1} 
\\
\hat g(X,Y)_p &=& 0 \quad \forall X \in \mathcal{H}_p, Y \in V_p, 
\label{eq:4_2}
\\
\hat g(\partial_\varphi,\partial_\varphi)_p &=& 1/K^2(\pi(p)). 
\label{eq:4_3}
\end{eqnarray} 

The tensor field $p \to \hat g_p$ defined on the open subset $SO(M,g) \setminus \hat \Sigma$ is called the \emph{Wagner lift} of the metric tensor field $g$. Obviously, $\hat g$ is a metric tensor field on  $SO(M,g) \setminus \hat \Sigma$. 
\end{definition}

\begin{remark}
The definition of values of $\hat g$ on the vertical vector fields follows from V.\,Wagner's construction of metric extension (see Introduction and the formula \eqref{eq:3_16}). 
\end{remark}

Let $U \subset M \setminus \Sigma$ be an open set, and let $\{e_1,e_2\}$ be a positively oriented orthonormal frame field on $U$.
From Definition~\ref{def:4_1} it follows that, for the frame field $\{E^h_1, E^h_2, \partial_\varphi\}$ defined on $\pi^{-1}(U)$, 
we have
\begin{eqnarray}
\label{eq:4_50}
\hat g(E^h_a,E^h_b) &=& g(e_a,e_b)=\delta_{ab},\quad a,b=\overline{1,2}, \notag
\\
\hat g(E^h_a,\partial_\varphi) &=& 0, 
\\
\hat g(\partial_\varphi,\partial_\varphi) &=& 1/K^2. \notag
\end{eqnarray} 

We will call the set $\hat \Sigma$ the \emph{singular set} of $\hat g$. In general, the set $\Sigma$ can be an almost arbitrary closed set, but we will consider only the case of ``general position'' when \emph{$\Sigma$ is a curve in $M$, and so $\hat \Sigma$ is a surface in $SO(M,g)$}.

Let us consider now the tensor $\hat g^*$ on $SO(M,g) \setminus \hat \Sigma$ dual to $\hat g^*$. From \eqref{eq:4_50} it follows that with respect to the frame field $\{E_1^h, E^2_h, \partial_\varphi\}$ the matrix of $\hat g$ can be written as follows:
\begin{equation*}
( \hat g) = 
\left(
\begin{array}{ccc}
1 & 0 & 0
\\
0 & 1 & 0
\\
0 & 0 & 1/K^2
\end{array}
\right),
\label{eq:4_51}
\end{equation*}
so the matrix of the dual tensor $\hat g^*$ is  
\begin{equation*}
( \hat g^* ) = 
\left(
\begin{array}{ccc}
1 & 0 & 0
\\
0 & 1 & 0
\\
0 & 0 & K^2
\end{array}
\right).
\label{eq:4_52}
\end{equation*}
From this follows that the tensor $\hat g^*$ can be extended smoothly to the entire manifold $SO(M,g)$ and one can easily see that this extension does not depend on a choice of the frame field $\{e_1,e_2\}$. Thus on the manifold $SO(M,g)$ we have a tensor field $\hat g^*$ of type $(0,2)$ which is the dual tensor field to the Wagner lift $\hat g$ on $SO(M,g) \setminus \hat\Sigma$ and is degenerate at $\hat\Sigma$.  

This gives rise to the following definition: \emph{a metric with singularities on a manifold $Q$ is a symmetric tensor field $h^*$ of type $(0,2)$ on $Q$, which is nondegenerate at $Q \setminus S$, where $S$ is a closed submanifold}. Certainly, the symmetric tensor $h$ of type $(2,0)$ dual to $h^*$ on $Q \setminus S$ determines a Riemannian metric on $Q \setminus S$.

\subsection{Relation between basic geometrical objects  of $(M,g)$ and $(SO(M,g),\hat g)$}
\label{subsec:2_6}
In this subsection we summarize the results concerning the relation between basic geometrical objects of $(M,g)$ and $(SO(M,g),\hat g)$, which were proved in \cite{Arteaga-Malakhaltsev}, and which we will use throughout this paper. 

Let $\{e_{1}(x), e_{2}(x)\}$ be an orthonormal frame defined on an open set $U\subset M \setminus \Sigma$. On $\pi^{-1}(U) \subset SO(M,g)$ we define the frame $\{\mathcal{E}_1,\mathcal{E}_2,\mathcal{E}_3\}$, which is orthonormal with respect to $\hat g$, as follows:
\begin{equation}
\begin{array}{l}\label{eq:4_18}
\mathcal{E}_1(x,\varphi) = E^h_1(x,\varphi) = e_1(x) + c^{\hphantom{1}1}_{12}(x) \partial_\varphi, \\
\mathcal{E}_2(x,\varphi) = E^h_2(x,\varphi) = e_2(x) + c^{\hphantom{1}2}_{12}(x) \partial_\varphi, \\
\mathcal{E}_3(x,\varphi) = K(x)\partial_\varphi,
\end{array}
\end{equation}
where $c^{\hphantom{1}c}_{ab}$ are the structure functions of the frame field $\{e_{1}(x), e_{2}(x)\}$ on $M$, and  $K(x)$ is the curvature at $x\in M$ of $(M,g)$.  

\begin{theorem}[\cite{Arteaga-Malakhaltsev}]\label{prop:5.1}
The structure functions $\hat{c}_{ij}^{\hphantom{1}k}$ of $\{\mathcal{E}_1,\mathcal{E}_2,\mathcal{E}_3\}$, $i,j,k=\overline{1,3}$,
and the structure functions $c_{ab}^{\hphantom{1}c}$ of $\{e_1, e_2\}$, $a,b,c=\overline{1,2}$ satisfy 
\begin{equation}
\begin{array}{lll}\label{eq:5_19}
\hat c^{\hphantom{1}1}_{12} = c^{\hphantom{1}1}_{12} & \hat c^{\hphantom{1}1}_{13} = 0 & \hat c^{\hphantom{1}1}_{23} = 0  
\\
\hat c^{\hphantom{1}2}_{12} = c^{\hphantom{1}2}_{12} & \hat c^{\hphantom{1}2}_{13} = 0 & \hat c^{\hphantom{1}2}_{23} = 0  
\\
\hat c^{\hphantom{1}3}_{12} = 1 & \hat c^{\hphantom{1}3}_{13} = \frac{e_1 K}{K} & \hat c^{\hphantom{1}3}_{23} = \frac{e_2 K}{K}  
\end{array}
\end{equation}
\end{theorem}

\begin{theorem}[\cite{Arteaga-Malakhaltsev}]\label{prop:5_2}
The connection coefficients of $\hat\nabla$ on $SO(M,g)$ are: 
\begin{equation*}
\begin{array}{lll}\label{eq:5_24}
\hat \Gamma^{\hphantom{1}1}_{12} = c^{\hphantom{1}1}_{12} & \hat \Gamma^{\hphantom{1}1}_{13} = 0 & \hat \Gamma^{\hphantom{1}2}_{13} = -\frac{1}{2} \\
\hat \Gamma^{\hphantom{1}1}_{22} = c^{\hphantom{1}2}_{12} & \hat \Gamma^{\hphantom{1}1}_{23} = \frac{1}{2} & \hat \Gamma^{\hphantom{1}2}_{23} = 0 \\
\hat \Gamma^{\hphantom{1}1}_{32} = \frac{1}{2} & \hat \Gamma^{\hphantom{1}1}_{33} = \frac{e_1 K}{K} & \hat \Gamma^{\hphantom{1}2}_{33} = \frac{e_2 K}{K}
\end{array}
\end{equation*}
The other connection coefficients are determined by the property $\hat \Gamma^{\hphantom{1}b}_{ac} = -\hat \Gamma^{\hphantom{1}c}_{ab}$.
\end{theorem}
Here $K=K(x)$ is the curvature of $(M,g)$, and $e_{i}K$ stands for the differentiation of $K$ with respect to  $e_{i}$. 

\begin{theorem}[\cite{Arteaga-Malakhaltsev}]\label{prop:5_3}
The coordinates $\hat{R}^{l}_{ijk}$ of the curvature tensor are:
\begin{equation*}\label{eq:5_26}
\begin{array}{l}
\hat R_{1212} = \frac{3}{4} - K,
\\
\hat R_{1213} = \frac{e_1 K}{K},
\\
\hat R_{1223} = \frac{e_2 K}{K},
\\
\hat R_{1313} = -\frac{1}{4} - e_1(\frac{e_1 K}{K}) - c^1_{12} \frac{e_2 K}{K} 
+ \left(\frac{e_1 K }{K}\right)^2,  
\\
\hat R_{1323} = - e_1(\frac{e_2 K}{K}) + c^1_{12} \frac{e_1 K}{K} 
+ \frac{e_1 K }{K}\frac{e_2 K }{K},
\\
\hat R_{2323} = -\frac{1}{4} - e_2(\frac{e_2 K}{K}) + c^2_{12} \frac{e_1 K}{K} 
+ \left(\frac{e_2 K }{K}\right)^2.  
\end{array}
\end{equation*}
\end{theorem}

\section{Isometries of the Wagner lift of metric}

Let $(M,g)$ be a two-dimensional oriented Riemannian manifold, and $(SO(M,g),\hat g)$ be the total space of the $SO(2)$-principal bundle $\pi : SO(M,g) \to M$ of positively oriented orthonormal frames of $(M,g)$ endowed by the Wagner lift $\hat g$ of the metric $g$. 

Let us denote by $L(M)$ the linear frame bundle of $M$.
Any diffeomorphism $f : M \to M$ induces the diffeomorphism $f^c : L(M) \to L(M)$, $\{e_1,e_2\}_x \mapsto \{df_x(e_1), df_x(e_2)\}$, where $df : TM \to TM$ is the tangent map. Generally, $f^c$ does not map $SO(M,g)$ to itself, but if $f$ is an isometry of $(M,g)$ preserving orientation, then $f^c$ maps an orthonormal frame to an orthonormal frame, so we have the map $f^c : SO(M,g) \to SO(M,g)$.
\emph{In what follows we consider isometries preserving orientation}.

If $\Sigma = \{x \in M | K(x) = 0\}$ is a one-dimensional submanifold, then, for any isometry $f$, we have $f(M \setminus \Sigma)=M \setminus \Sigma$. Therefore, the lift $f^c$  determines a diffeomorphism $f^c : SO(M,g) \setminus \hat{\Sigma} \to  SO(M,g) \setminus \hat{\Sigma}$.

\begin{theorem} \label{prop:5_4}
If $f$ is an isometry of $(M,g)$, then $f^c$ is an isometry of $(SO(M,g) \setminus \hat{\Sigma},\hat g)$. 
\end{theorem}
\begin{proof} For any $p \in SO(M,g)$ and $\mu \in SO(2)$, we have $f^c(p \mu) = f^c(p) \mu$, hence $f^c$ maps the fibers of $SO(M,g)$ onto the fibers, and the fundamental vector field $\partial_\varphi$ is invariant under $f^c$. Any isometry $f$ of the metric $g$ is an automorphism of the Levi-Civita connection of $g$, therefore $f^c$ is an automorphism of the infinitesimal connection $\mathcal{H}$, so maps $\mathcal{H}(p)$ to $\mathcal{H}(f^c(p))$ for any $p \in SO(M,g)$. Hence, for any $v \in T_x M$, and $p = \pi^{-1}(x)$, the horizontal lift $v^h$ of $v$ at $p$ is mapped by $f^c$ to the horizontal lift of the vector $df_x(v)$ at the point $f^c(p)$, i.e.  $df^c_p (v^h) = (df_x(v))^h$. Finally, for the curvature we have $K(f(x))=K(x)$ for any $x \in M$.  

For each orthonormal frame field $\{e_1,e_2\}$ on an open set $U \subset M$, we have constructed the  frame field $\{\mathcal{E}_1,\mathcal{E}_2,\mathcal{E}_3\}$ on $\pi^{-1}(U) \subset SO(M,g)$ (see \eqref{eq:4_18}) which is orthonormal with respect to $\hat g$. Let $U' = f(U)$. Then $f$ maps the orthonormal frame field $\{e_1,e_2\}$ on $U$ to an orthonormal frame field $\{e'_1,e'_2\}$ on $U'$. From above it follows that the complete lift $f^c$ maps the vector fields $\mathcal{E}_a$, $a=\overline{1,3}$, to the vector fields  $\mathcal{E}'_a$ which form the orthonormal frame field corresponding to $\{e'_1,e'_2\}$. Hence $f^c$ is an isometry.
\end{proof}
\begin{remark}
Theorem~\ref{prop:5_4} says in fact that the Wagner lift is a ``natural'' construction, that is the correspondence $(M,g) \to (SO(M,g),\hat g)$ is a functor from the category of the two-dimensional Riemannian manifolds (the morphisms of the category are isometries) to the category of the three-dimensional Riemannian manifolds.
\end{remark}

Any isometry $f$ maps $M \setminus \Sigma$ onto itself, and is an isometry of $(M \setminus \Sigma, g|_{M \setminus \Sigma})$. The lift $f^c$ also maps  $SO(M,g) \setminus \hat{\Sigma}$ onto itself and, according to Theorem~\ref{prop:5_4}, is an isometry of $\hat g$. Therefore, in this section, for the brevity of notation, \emph{we will assume that the curvature $K$ of $g$ does not vanish}. The reader can easily obtain the versions of results below for the case $\Sigma \ne \emptyset$ by substituting $M \setminus \Sigma$ for $M$, and $SO(M) \setminus \hat\Sigma$ for $SO(M,g)$. 

Let us denote by $I(N,h)$ the isometry group of a Riemannian manifold $(N,h)$. 
\begin{corollary}
The map   $c : I(M,g) \to I(SO(M,g),\hat g)$, $f \to f^c$, is a monomorphism of the isometry groups.
\end{corollary}

We say that a diffeomorphism $\hat f : SO(M,g) \to SO(M,g)$ is projectable if there exists a diffeomorphism $f : M \to M$ such that $\pi \circ \hat f = f \circ \pi$, and we write $f = \pi_*(\hat f)$.
Let us denote by $I_p(SO(M,g),\hat g)$ the Lie group of projectable isometries of $(SO(M,g), \hat g)$.

\begin{theorem}
\label{prop:5_4_1}
For any $f \in I_p(SO(M,g),\hat g)$, $\pi_*(f) \in I(M,g)$.  
\end{theorem}
\begin{proof}
We have $\pi \circ \hat f = f \circ \pi$, hence follows that $\hat f$ maps each fiber of the bundle $SO(M,g) \to M$ onto a fiber. At the same time, as the horizontal distribution $\mathcal{H}$ is orthogonal to the fibers, and $\hat f$ is an isometry, then $\hat f$ is  an automorphism of the distribution $\mathcal{H}$. Therefore, for any point $x \in M$ and $V, W \in T_x M$, we take $p \in \pi^{-1}(x)$ and $\hat V, \hat W \in \mathcal{H}(p)$ such that $d\pi_p(\hat V)=V$ and  $d\pi_p(\hat W)=W$. Set $V' = df_x(V)$, $W' = df_x(W)$,   $\hat V' = d\hat f_p(\hat V)$, $\hat W' = d\hat f_p(\hat W)$. 
Then we have 
\begin{equation}
g_{x}(V,W) = \hat g_{p}(\hat V, \hat W) = \hat g_{f(p)}(\hat V', \hat W') = g_{f(x)}(V',W').
\label{} 
\end{equation}
This proves the requirement statement.
\end{proof}

It is clear that $\pi_* : I(SO(M,g),\hat g) \to I(M,g)$ is a Lie group homomorphism, and $\pi_* \circ c = Id_{I(M,g)}$, so we obtain the following result. 
\begin{corollary}
The Lie group homomorphism $\pi_* : I(SO(M,g),\hat g) \to I(M,g)$ is surjective.
\end{corollary}

Recall that an infinitesimal isometry of a Riemannian manifold $(N,h)$ is a vector field whose flow consists of the isometries of $(N,h)$. The set of infinitesimal isometries of $(N,h)$ is a Lie subalgebra in the algebra of vector fields on $N$, and  we will denote it by $\mathfrak{I}(N,h)$. Note that $\mathfrak{I}(N,h)$ is the Lie algebra of the Lie group $I(N,h)$. 

Let $\phi_s$ be the flow of an infinitesimal isometry $X \in \mathfrak{I}(M,g)$, then, by Theorem~\ref{prop:5_4}, the flow $\phi^c_s$ on $SO(M,g)$ consists of isometries of $(SO(M,g),\hat g)$, so the \emph{complete lift} $X^c = \frac{d}{ds}|_{s=0} \phi_s^c$ of $X$ is an infinitesimal isometry in $\mathfrak{I}(SO(M,g),\hat g)$. Thus we get the 

\begin{corollary}
The map   $c_* : \mathfrak{I}(M,g) \to \mathfrak{I}(SO(M,g),\hat g)$, $X \mapsto X^c$, is a monomorphism of the Lie algebras of the isometry groups.
\end{corollary}

Recall that a vector field $\tilde X$ on the total space $E$ of a fiber bundle $\pi : E \to B$ is called \emph{projectable} if $\tilde X$ is $\pi$-connected to a vector field $X$ on $B$.   It is clear that any $SO(2)$-invariant vector field $\tilde X$ on $SO(M,g)$ is projectable, the opposite is not true. The set $\mathfrak{X}_p(SO(M,g))$ of projectable vector fields is a Lie subalgebra in the Lie algebra $\mathfrak{X}(SO(M,g))$ of vector fields on $SO(M,g)$, and since the isometries also form a Lie subalgebra of $\mathfrak{X}(SO(M,g))$, we arrive at \emph{the Lie subalgebra $\mathfrak{I}_p(SO(M,g),\hat g) \subset \mathfrak{X}(SO(M,g))$ of projectable isometries}.  In what follows we will describe the relationship between $\mathfrak{I}_p(SO(M,g),\hat g)$ and $\mathfrak{I}(M,g)$.  

From the definition of projectable vector field we get the natural Lie algebra homomorphism $\pi_* : \mathfrak{I}_p(SO(M,g),\hat g) \to \mathfrak{X}(M)$ which assigns to each $X \in \mathfrak{I}_p(SO(M,g),\hat g)$ the vector field $\pi_*(X) \in \mathfrak{X}(M)$ which is $\pi$-connected with $X$. 

\begin{theorem}
For any $X \in \mathfrak{I}_p(SO(M,g),\hat g)$, the vector field $\pi_*(X)$ belongs to $\mathfrak{I}(M,g)$. 
\label{thm:3_2}
\end{theorem}
\begin{proof}
Let $X \in \mathfrak{I}(SO(M,g),\hat g)$ be a projectable vector field which is $\pi$-connected with a vector field $Y$ on $M$, and let us denote the flows of $X$ and $Y$ by $\phi^X_s$ and $\phi^Y_s$ respectively.
Then we have $\pi \phi^X_s = \phi^Y_s \pi$, hence Theorem~\ref{prop:5_4_1} implies the result.
\end{proof}

\begin{corollary}
$\pi_* : \mathfrak{I}_p(SO(M,g),g) \to \mathfrak{I}(M,g)$ is a surjective Lie algebra homomorphism. 
\end{corollary}

A vector field $X$ is an infinitesimal isometry of $(SO(M,g),\hat g)$ if and only if $\mathcal{L}_X \hat g = 0$, where $\mathcal{L}_X$ is the Lie derivative. 
The equation $\mathcal{L}_X \hat g = 0$ can be written as 
\begin{equation*}
X \hat g(Y,Z) - \hat g([X,Y],Z) - \hat g(Y,[X,Z]) = 0,
\label{eq:5_3_1}
\end{equation*}
for any vector fields $Y$ and $Z$.  Now, if we take the orthonormal frame field $\{\mathcal{E}_i\}$ \eqref{eq:4_18},  then this equation can be written, for a vector field $X = X^i \mathcal{E}_i$, as follows: 
\begin{equation*}
\mathcal{E}_i X^j + \mathcal{E}_j X^i + (\hat c^{\hphantom{1}j}_{is} + \hat c^{\hphantom{1}i}_{js})X^s = 0. 
\label{eq:5_3_2}
\end{equation*}  
Hence for an infinitesimal isometry 
$X \in \mathfrak{I}(SO(M),\hat g)$ we get the following system of differential equations:
\begin{eqnarray*}
\mathcal{E}_1 X^1 + \hat c^{\hphantom{1}1}_{1s}X^s = 0,
\\
\mathcal{E}_2 X^2 + \hat c^{\hphantom{1}2}_{2s}X^s = 0,
\\
\mathcal{E}_3 X^3 + \hat c^{\hphantom{1}3}_{3s}X^s = 0,
\\
\mathcal{E}_1 X^2 + \mathcal{E}_2 X^1 + (\hat c^{\hphantom{1}2}_{1s} + \hat c^{\hphantom{1}2}_{1s})X^s = 0,
\\
\mathcal{E}_1 X^3 + \mathcal{E}_3 X^1 + (\hat c^{\hphantom{1}3}_{1s} + \hat c^{\hphantom{1}1}_{3s})X^s = 0,
\\
\mathcal{E}_2 X^3 + \mathcal{E}_3 X^2 + (\hat c^{\hphantom{1}3}_{2s} + \hat c^{\hphantom{1}2}_{3s})X^s = 0.
\label{eq:5_3_3}
\end{eqnarray*}
If now we substitute $\hat c^k_{ij}$ from \eqref{eq:5_19}, we get that 
\begin{eqnarray}
\mathcal{E}_1 X^1 +  c^{\hphantom{1}1}_{12}X^2 = 0,
\label{eq:5_3_4_1}
\\
\mathcal{E}_2 X^2 + \hat c^{\hphantom{1}2}_{21}X^1 = 0,
\label{eq:5_3_4_2}
\\
\mathcal{E}_3 X^3 - \frac{e_1 K}{K} X^1 - \frac{e_2 K}{K} X^2 = 0,
\label{eq:5_3_4_3}
\\
\mathcal{E}_1 X^2 + \mathcal{E}_2 X^1 + c^{\hphantom{1}1}_{21} X^1 + c^{\hphantom{1}2}_{12} X^2 = 0,
\label{eq:5_3_4_4}
\\
\mathcal{E}_1 X^3 + \mathcal{E}_3 X^1 - X^2 + \frac{e_1 K}{K} X^3 = 0,
\label{eq:5_3_4_5}
\\
\mathcal{E}_2 X^3 + \mathcal{E}_3 X^2 - X^1 + \frac{e_2 K}{K} X^3 = 0.
\label{eq:5_3_4_6}
\end{eqnarray}

\begin{theorem}
a) If an infinitesimal isometry $X$ of $(SO(M,g),\hat g)$ is vertical, then $X = C \partial_\varphi$, where $C$ is a constant.

b) Any projectable infinitesimal isometry $X$ of $(SO(M,g),\hat g)$ can be written as $\tilde X = Y^c + C\partial_\varphi$, where $Y$ is the infinitesimal isometry of $(M,g)$ which is $\pi$-connected with $X$, and $C$ is a constant.

c) The Lie subalgebra ${\frak I}_p(SO(M,g),\hat g)$ of projectable infinitesimal isometries is isomorphic to the direct sum of the Lie algebra $\mathfrak{I}(M,g)$ and the one-dimensional Lie algebra $\mathbb{R}$.     
\end{theorem}
\begin{proof}
a) Let $X \in \mathfrak{I}(SO(M,g),\hat g)$ be a vertical vector field.  Then the equations \eqref{eq:5_3_4_1}, \eqref{eq:5_3_4_2}, \eqref{eq:5_3_4_4} hold true as $X^1=X^2=0$, and from $\eqref{eq:5_3_4_3}, \eqref{eq:5_3_4_5}, \eqref{eq:5_3_4_6}$ we get that $X = C K \mathcal{E}_3 =  C \partial_\varphi$.

b) Now suppose that an infinitesimal isometry $X$ of $(SO(M,g),\hat g)$ is projectable, then $Y = \pi_*(X)$ is an infinitesimal isometry of $(M,g)$ by Theorem~\ref{thm:3_2}.  Then $Z = X - Y^c$ is a vertical infinitesimal isometry of $(SO(M,g),\hat g)$, hence $Z=C\partial_\varphi$, by a). 

c) From b) it follows that $\mathfrak{I}_p(SO(M,g),\hat g)$ is the sum of 
$c_*(\mathfrak{I}(M,g)) \cong \mathfrak{I}(M,g)$ and the Lie algebra $\mathfrak{I}_v(SO(M,g),\hat g)$ of vertical isometries which is isomorphic to $\mathbb{R}$. It is clear that $\mathfrak{I}_v(SO(M,g),\hat g) \cap c_*(\mathfrak{I}(M,g)) = \{0\}$, and moreover, for any $Z \in   \mathfrak{I}_v(SO(M,g),\hat g)$ and $X \in c_*(\mathfrak{I}(M,g))$, the Lie bracket $[X,Z] = 0$. This implies the required statement. 
\end{proof}

\begin{corollary}
$\dim I(SO(M,g),\hat g) \ge \dim I(M,g) + 1$.
\end{corollary}

\begin{corollary}
If $I(M,g)$ is transitive, then $I(SO(M,g),\hat g)$ is transitive, too.  
\end{corollary}

\begin{remark}
It is known, that the isometry group of a simply connected Riemannian manifold $(M,g)$ of constant sectional curvature acts transitively on $M$ and moreover the induced action on the total space of orthonormal frames is simply transitive. 
From this follows that, for such a manifold, the total space of orthonormal frames is diffeomorphic to the Lie group $I(M)$ of the isometries of $(M,g)$. 
From Theorem~\ref{prop:5_4} it follows that for a simply connected two-dimensional $(M,g)$ of constant nonzero curvature, the $I(M,g)$ acts on $SO(M,g) \cong I(M,g)$ by the isometries with respect to the Wagner lift metric $\hat g$. Thus, $(SO(M,g),\hat g)$ is a Lie group endowed with a left-invariant metric.        
If $(M,g)$ is the sphere, then $SO(M,g)$ is diffeomorphic to $SO(3)$ with the left-invariant metric of constant curvature induced by the double covering $\mathbb{S}^3 \to SO(3)$. If $(M,g)$ is the hyperbolic plane, the metric $\hat g$ does not have constant curvature (\cite{Arteaga-Malakhaltsev}, Sec.~6).   
\end{remark}

\label{sec:3}
\section{Geodesics of the Wagner lift and the singular set}
\label{sec:4}
\subsection{Geodesics of the Wagner lift and their projections}
\label{subsec:4-1}

In \cite{Arteaga-Malakhaltsev} we have proved the following statement which establish relation between geodesics of $(M,g)$ and $(SO(M,g),\hat g)$: 
\begin{theorem}[\cite{Arteaga-Malakhaltsev}, Theorem 5.3]
\label{thm:5_3_0}
Let  $\hat\gamma(t)$ be a  geodesic of the  metric $\hat g$ on $SO(M,g)$, $\gamma(t)$  its projection on $M$, $t \in [0,a]$. Then
\begin{enumerate}
\item 
There holds the equation:
\begin{equation}
\hat g\left(\mathcal{E}_3,\frac{d\hat\gamma}{dt}(t)\right) = C K(\gamma(t)), \text{where $C$ is constant.}\notag
\end{equation}
\item 
The curve $\gamma$ satisfies the differential equation 
\begin{equation}
\nabla_{\frac{d\gamma}{dt}}\frac{d\gamma}{dt} = C K J (\dot\gamma) - C^2 K grad K, 
\label{eq:2_28}
\end{equation} 
where $J$ is the operator of the complex structure on $M$ associated with the metric $g$.
\item 
If $\frac{d\hat\gamma}{dt}(t)$ is horizontal at $t_0$, then  $\hat \gamma$ is a horizontal curve, that is $\hat\gamma$ is tangent to $\mathcal{H}$ for all  $t$,
and  $\gamma$ is a geodesic of the metric $g$ on $M$.
\end{enumerate}
\end{theorem}
 
The proof of this theorem is based on the following fact which was also proved in 
\cite{Arteaga-Malakhaltsev}:
\begin{theorem}[\cite{Arteaga-Malakhaltsev}, Theorem~5.2]
\label{thm:5_2_0}
Let $\hat\nabla$ be the Levi-Civita connection of $\hat g$.
Let $\hat\gamma$ be a geodesic of the connection  $\hat \nabla$ on $SO(M,g)$, and  
\begin{equation}
\frac{d}{dt}\hat\gamma(t) = Q^i(t) \mathcal{E}_i |_{\gamma(t)}. \notag
\end{equation}
Then the functions   $Q^i(t)$ satisfy the equations 
\begin{eqnarray}
\frac{dQ^1}{dt} + c^{\hphantom{1}1}_{12} Q^1 Q^2 + c^{\hphantom{1}2}_{12} (Q^2)^2 + Q^2 Q^3 + \frac{e_1 K}{K}(Q^3)^2 = 0,
\label{eq:5_28}
\\ 
\frac{dQ^2}{dt} - c^{\hphantom{1}1}_{12} (Q^1)^2 - c^{\hphantom{1}2}_{12} Q^1 Q^2 - Q^1 Q^3 + \frac{e_2 K}{K}(Q^3)^2 = 0,
\label{eq:5_29}
\\
\frac{dQ^3}{dt} - \frac{e_1 K}{K} Q^1 Q^3 - \frac{e_2 K}{K} Q^2 Q^3  = 0.
\label{eq:5_30}
\end{eqnarray} 
\end{theorem}

In the present section we will establish how the geodesics of $(SO(M,g),\hat g)$ behave with respect to the singular set $\hat\Sigma$ (see subsection~\ref{subsec:2_4}). Recall that, under our assumption, $\Sigma$ is a one-dimensional submanifold of  $M$, so $\hat\Sigma$ is a two-dimensional submanifold in $SO(M,g)$. We start with the following statement which is in some sense converse to Theorem~\ref{thm:5_3_0}. 

\begin{theorem}
\label{thm:3_3}
Let $\gamma : I=(a,b) \to M$ be a solution for the equation \eqref{eq:2_28}. Let $I_1 = \gamma^{-1}(M \setminus \Sigma)$. Then, for each $t_0 \in I$ and for each  $p \in \pi^{-1}(\gamma(t_0))$, there exists a curve $\hat\gamma(t)$ such that $\hat\gamma(t_0)=p$,  $\pi(\hat\gamma(t))=\gamma(t)$, and $\hat\gamma(t)$ restricted to the open set $I_1$ is a geodesic of the metric $\hat g$ on $SO(M,g) \setminus \hat \Sigma$. 
\end{theorem}
\begin{proof}
For any $p \in \pi^{-1}(t_0)$ we will present a curve $\hat\gamma : (a,b) \to SO(M,g)$ with the required properties. Let $\gamma^h(t)$ be the horizontal lift of $\gamma(t)$ such that $\gamma^h(t_0) = p$, and set $\alpha(t)=C \int^t_{t_0} K^2(\gamma(s))ds$, where $C$ is the constant in \eqref{eq:2_28}. We set 
\begin{equation}
\hat\gamma(t) = \gamma^h(t) r(\alpha(t)).  
\end{equation}
It is clear that $\pi(\hat\gamma(t)) = \gamma(t)$ and $\hat\gamma(t_0) = p$. We will prove that restricted to $I_1$ the curve $\hat\gamma(t)$ is a geodesic of $\hat g$. 

First note that $I_1$ is open and $\hat\gamma : I_1 \to SO(M,g)\setminus \hat\Sigma$.
We will use the following well-known statement: \emph{If $P \to M$ is a principal $G$-bundle, $\delta : (a,b) \to P$, $g : (a,b) \to G$, then }
\begin{equation*}
\frac{d}{dt} (\delta(t)g(t)) = dR_{g(t)}\left(\frac{d\delta}{dt}\right) + \sigma\left(dL_{g^{-1}(t)}\frac{dg}{dt}\right),  
\label{eq:3_100}
\end{equation*}
\emph{where $\sigma(a)$ is the fundamental vector field on $P$ corresponding to the vector $a$ in the Lie algebra $\mathfrak{g}$ of the Lie group $G$}.
Note that in our case, $P = SO(M,g)$,   $G = SO(2)$,  and if we take $g(t) = r(\alpha(t))$, then one can easily find that 
\begin{equation*}
\sigma\left(dL_{g^{-1}(t)}\frac{dg}{dt}\right) = \frac{d\alpha}{dt} \partial_\varphi = C K^2(\gamma(t)) \partial_\varphi. 
\label{eq:3_101}
\end{equation*}
From this follows that 
\begin{equation}
\frac{d}{dt}\hat\gamma = dR_{r(\alpha(t))}\left(\frac{d\gamma^h}{dt}\right) + C K^2(\gamma(t))\partial_\varphi.
\label{eq:3_103}
\end{equation}
Hence follows that $Q^3 = C K(\gamma(t))$.

Without loss of generality we may assume that the curve $\gamma : I_1 \to M \setminus \Sigma$ takes values in an open set $U$ on which an orthonormal frame field $\{e_1,e_2\}$ exists, and so on $\pi^{-1}(U)$ there is the orthonormal frame field $\{\mathcal{E}_i\}$, $i=\overline{1,3}$, defined by \eqref{eq:4_18}.  
Then 
\begin{equation*}
\frac{d}{dt}\gamma^h = Q^1(t) \mathcal{E}_1|_{\gamma^h(t)} + Q^2(t) \mathcal{E}_2|_{\gamma^h(t)}.
\label{eq:3_104}
\end{equation*}
Since $\gamma(t)=\pi\gamma^h(t)$ satisfies \eqref{eq:2_28} and 
\begin{equation*}
\frac{d}{dt}\gamma = Q^1(t) e_1|_{\gamma^h(t)} + Q^2(t) e_2|_{\gamma^h(t)},
\label{eq:3_105}
\end{equation*}
we get that 
\begin{eqnarray}
\frac{dQ^1}{dt} + c^{\hphantom{1}1}_{12} Q^1 Q^2 + c^{\hphantom{1}2}_{12} (Q^2)^2 + C K Q^2  + C^2 K e_1 K = 0,
\\ 
\frac{dQ^2}{dt} - c^{\hphantom{1}1}_{12} (Q^1)^2 - c^{\hphantom{1}2}_{12} Q^1 Q^2 - C K Q^1  + C^2 K e_2 K = 0.
\label{eq:3_106}
\end{eqnarray} 
The vector fields $\mathcal{E}_1$, $\mathcal{E}_2$ are invariant under right shifts, therefore we have that
\begin{equation*}
\frac{d}{dt}\hat\gamma = 
Q^1(t) \mathcal{E}_1|_{\hat\gamma(t)} + Q^2(t) \mathcal{E}_2|_{\hat\gamma(t)} + C K^2(\gamma(t))\partial_\varphi.
\label{eq:3_107}
\end{equation*}
Now one can easily check that the equations \eqref{eq:5_28}--\eqref{eq:5_30} hold true, therefore $\hat\gamma(t)$  is a geodesic for $t \in I_1$. 
\end{proof}

\begin{corollary}
Let $\gamma: I=(a,b) \to M$ be a curve. 
If $\gamma(t)$ is  a geodesic of $(M,g)$, then, for any $t_0 \in I_1$ and any $p \in \pi^{-1}(\gamma(t_0))$, the horizontal lift $\gamma^h(t)$ passing through $p$ is a geodesic of the metric $\hat g$ on $SO(M,g) \setminus \hat\Sigma$ for  $t \in I_1$. 
\end{corollary}
\begin{proof}
The curve $\gamma(t)$ satisfies the equation \eqref{eq:2_28} with $C=0$. 
Then for the geodesic $\hat\gamma(t)=\gamma^h(t)r(\alpha(t))$ constructed in the proof of Theorem~\ref{thm:3_3}  we have $\alpha(t)=0$, hence follows that $\hat\gamma(t) = \gamma^h(t)$ for $t \in I$. 
\end{proof}

\begin{definition}
A curve $\hat\gamma : (a,b) \to SO(M,g)$ is called \emph{a generalized geodesic} of the metric $\hat g$  if $\hat\gamma(t)$ is a geodesic for $t \in \hat\gamma^{-1}(SO(M,g)\setminus \hat{\Sigma})$. 
\end{definition}

\begin{corollary}
For any $p \in \hat\Sigma$, and for any vector $V \in \mathcal{H}_p$, there exists a generalized geodesic $\hat\gamma(t)$, $t \in (-\varepsilon,\varepsilon)$, of the metric $\hat g$ such that $\hat\gamma(0)=p$ and $\frac{d\hat\gamma}{dt}(0)=V$.
\end{corollary}
\begin{proof}
Let $x = \pi(p)$ and $W = d\pi_p(V)$. Find a solution $\gamma$ of \eqref{eq:2_28} defined on $(-\varepsilon,\varepsilon)$ such that $\gamma(0)=x$ and $\frac{d\hat\gamma}{dt}(0)=W$. Now apply Theorem~\ref{thm:3_3} to $\gamma$ and $p$, then we get that  there exists a generalized geodesic $\hat\gamma(t)$ such that $\hat\gamma(0)=p$ and $\pi(\hat\gamma(t))=\gamma(t)$. Finally \eqref{eq:3_103} implies that $\frac{d\hat\gamma}{dt}(0)$ is horizontal, therefore $\frac{d\hat\gamma}{dt}(0)=V$.
\end{proof}

\begin{theorem}
\label{prop:4-1}
Any generalized geodesic $\hat\gamma(t)$ transversal to $\hat{\Sigma}$ intersects $\hat{\Sigma}$ horizontally, this means that, if $\hat\gamma(t_0)$ lies in $\hat\Sigma$, then $\frac{d}{dt}|_{t=t_0}\hat\gamma(t)$ lies in $\mathcal{H}_{\hat\gamma(t_0)}$.
\end{theorem}
\begin{proof}
Let $\hat\gamma : I = (a,b) \to SO(M)$ be a generalized geodesic of $(SO(M,g),\hat g)$,  
$\gamma = \pi\hat{\gamma}$, and $I_1 = \gamma^{-1}(M \setminus \Sigma) = \hat\gamma^{-1}(SO(M) \setminus \hat\Sigma$. We have 
\begin{equation*}
\frac{d\hat{\gamma}}{dt}(t) = X_h(t) + f(t)\partial_\varphi,
\label{eq:3_117}
\end{equation*}
where $X_h(t)$ is a horizontal vector field and $f(t)$ is a function. 
From Theorem~\ref{thm:5_3_0} (1), using \eqref{eq:4_18}, we get that, for any $t \in I_1$,  
\begin{equation*}
\hat g(\partial_\varphi, \frac{d\hat\gamma}{dt}(t)) = 
\frac{1}{K(\gamma(t))} \hat g(\mathcal{E}_3,\frac{d\hat\gamma}{dt}(t)) = C,
\label{eq:3_115}
\end{equation*}
and, from \eqref{eq:4_3} we get that along $\hat \gamma(t)$ there holds the equality:    
\begin{equation*}
\hat g(\partial_\varphi, \partial_\varphi) =  \frac{1}{K^2(\gamma(t))}. 
\label{eq:3_116}
\end{equation*}
Let $\hat{\gamma}(t)$ intersects $\hat\Sigma$ transversally at $t_0$. Then, there exists $\varepsilon$ such that 
\begin{equation*}
f(t) = \frac{\hat g(\partial_\varphi, \frac{d\hat\gamma}{dt}(t))}{\hat g(\partial_\varphi, \partial_\varphi)} = C K^2(\gamma(t)) 
\label{eq:3_118}
\end{equation*}
holds for all $t \in (t_0 - \varepsilon,t_0 + \varepsilon)$ except for $t_0$. Since both sides are continuous on $(t_0 - \varepsilon,t_0 + \varepsilon)$, we get that $f(t_0)=C K^2(\gamma(t_0)) = 0$, this proves the theorem. 
\end{proof}

\subsection{Lagrangian for solutions of \eqref{eq:2_28}. Geodesic modeling }
\label{subsec:4-2}
Let $(M,g)$ be a Riemannian manifold, $F$ be a closed 2-form on $M$, $F^\#$ be the affinor field such that $g(F^\#X,Y)=F(X,Y)$, and $U$ be a function on $M$. 
Consider  the equation 
\begin{equation}
\nabla_{\frac{d\gamma}{dt}}\frac{d\gamma}{dt} = F^\# \dot\gamma  + \mathop{grad} U, 
\label{eq:3_200}
\end{equation}
where $\nabla$ is the Levi-Civita connection of $g$.  Equations of this type arise in physics, e.\,g., in the theory of the motion of charged particles in electromagnetic fields.  In (\cite{Novikov}, Sect.~5) S.P.\,Novikov constructed a multi-valued Lagrangian whose Euler-Lagrange equations are exactly the equations \eqref{eq:3_200} and developed an analog of the Morse theory for Lagrangians of this type. The theory of geodesic simulation for equations \eqref{eq:3_200} with $U = 1/f$, where $f$ is a nonvanishing function on $M$, was developed in the papers of Ya.L.\,Shapiro, V.A.\,Igoshin, and E.I.\,Yakovlev (one of the first papers devoted to this theory was \cite{ShapiroIgoshinYakovlev}, for more recent results in this direction we refer the reader to \cite{Igoshin1}, \cite{Igoshin2}, \cite{Yakovlev}).  Our equation \eqref{eq:2_28}  is a very partial case of the equation \eqref{eq:3_200}. In this subsection we will  write down the Lagrangian in our partial case and also demonstrate how to apply Shapiro-Igoshin-Yakovlev results to \eqref{eq:2_28}. 

First we repeat S.P.\,Novikov's considerations (\cite{Novikov}, Sect.~5) in our case. We consider the $2$-form $F = K \sigma$, where $\sigma$ is the area $2$-form of the metric tensor field $g$. Note, that, by the Gauss-Bonnet theorem, $F$ represents the Euler class of $M$, so for a compact oriented $M$, the $2$-form $F$ is exact if and only if $M$ is the two-dimensional torus $\mathbb{T}^2$.  If $M$ is not compact, $F$ is exact. 

Consider two points $x_0$ and $x_1$ in $M$ and the space $\Omega$ of paths joining these points and lying inside a region $C \subset M$ such that $F = d\theta$ on $C$. Take the action
\begin{equation}
\mathcal{L}(\gamma) = \int^a_b \frac{1}{2}\left[ g_{\gamma(t)}(\frac{d\gamma}{dt},\frac{d\gamma}{dt}) - C K(\gamma(t)) \theta_{\gamma(t)}(\frac{d\gamma}{dt}) + C^2 K^2(\gamma(t))\right] \,dt.
\label{eq:3_201}
\end{equation} 
Written with respect to local coordinates $(x^i,\dot x^i)$ on $TM$, the Lagrangian of the action has the form
$L = L_0 - L_1 - L_2$, where 
\begin{equation}
L_0(x^k, \dot x^k) =   \frac{1}{2} g_{ij}(x^k) \dot x^i \dot x^j, \quad
L_1(x^k, \dot x^k) = \frac{1}{2} C K(x^k) \theta_i(x^k) \dot x^i, \quad
L_2(x^k) =  - \frac{1}{2}  C^2 K^2(x^k).
\label{eq:3_202}
\end{equation}
The left hand of the Euler-Lagrange equation for $L_0$ is well known to be  
\begin{equation*}
\frac{d}{dt}(\frac{\partial L_0}{\partial \dot x^i}) - \frac{\partial L_0}{\partial x^i} = g_{is}\nabla_{\dot\gamma} \dot\gamma^s.
\label{eq:3_203}
\end{equation*} 
One can easily calculate that 
\begin{multline*}
\frac{d}{dt}(\frac{\partial L_1}{\partial \dot x^i}) - \frac{\partial L_1}{\partial x^i} = 
\frac{1}{2}(\frac{d}{dt}[\theta_i(\gamma(t))] - \partial_i\theta_j(\gamma(t)) \dot \gamma^j) = 
\\
\frac{1}{2}( \partial_j\theta_i(\gamma(t)) - \partial_i\theta_j(\gamma(t))) \dot \gamma^j) = C K(\gamma(t)) \sigma_{ji}\dot\gamma^j,
\label{eq:3_204}
\end{multline*}
and 
\begin{equation*}
\frac{d}{dt}(\frac{\partial L_2}{\partial \dot x^i}) - \frac{\partial L_2}{\partial x^i} = 
- C^2 K(\gamma(t)) \partial_i K(\gamma(t)) 
\label{eq:3_205}
\end{equation*}
Hence follows that the Euler-Lagrange equations for $L$ are 
\begin{equation*}
g_{is}\nabla_{\dot\gamma} \dot\gamma^s - C K \sigma_{ji}\dot \gamma^j + C^2 K \partial_i K = 0.
\label{eq:3_206}
\end{equation*}
Now note that $g^{mi}\sigma_{ji} = J^m_j$ and $\mathop{grad} K = g^{mi} \partial_i K \partial_m$. Thus, the Euler-Lagrange equations for the Lagrangian \eqref{eq:3_202} are 
\begin{equation*}
\nabla_{\dot\gamma} \dot\gamma - C K J(\dot \gamma) + C^2 K \mathop{grad} K = 0,
\label{eq:3_207}
\end{equation*}
which coincide with the equations \eqref{eq:2_28}. 
Thus we have proved that \emph{the action \eqref{eq:3_201} is a multivalued functional whose extremals are the solutions to the equation \eqref{eq:2_28}}.

Now we will briefly describe the construction of geodesic modeling  from \cite{ShapiroIgoshinYakovlev}. We consider the equation \eqref{eq:3_200}. Suppose that the form $F$ satisfies the ``quantization condition'': this means that there exists a certain $a \in \mathbb{R}$ such that for any cycle $C \subset M$, $\int_C F \in a \mathbb{Z}$. Then there exists a $SO(2)$-principal bundle $\pi : P \to M$ with the characteristic class $[F] \in H^2(M;\mathbb{R})$ (recall that, for any connection $1$-form $\omega$ on $P$, the curvature $2$-form $\Omega$ is the pull-back of a closed 2-form $F$ on $M$, and the cohomology class of $F$ does not depend on the choice of the connection and is called the characteristic class of the bundle). In \cite{ShapiroIgoshinYakovlev}  the following fact is proved: if on $P$ we take the metric  
\begin{equation}
\bar g_p (\bar X,\bar Y) = g(d\pi \bar X,  d\pi \bar Y) + f(\pi(p)) \omega(\bar X) \omega(\bar Y),
\label{eq:3_208}
\end{equation}   
where $p \in P$, $\bar X, \bar Y \in T_p P$, then the geodesics of $(P,\bar g)$ are projected onto the solutions of the equation \eqref{eq:3_200}, that is $(P,\bar g,\pi)$ is a ``Riemannian geodesic model'' of this equation.   

Now consider our equation \eqref{eq:2_28}. The 2-form $F = K \sigma$ is the characteristic class of the bundle $SO(M,g)$, this means that in our case $P = SO(M,g)$. Also $1/f = K^2$, hence follows that \eqref{eq:3_208} turns into  
\begin{equation*}
\bar g_p (\bar X,\bar Y) = g(d\pi \bar X,  d\pi \bar Y) + \frac{1}{K^2}(\pi(p)) \omega(\bar X) \omega(\bar Y),
\label{eq:3_209}
\end{equation*}   
which proves that $\bar g = \hat g$. Hence we conclude that \emph{for a two-dimensional Riemannian manifold $(M,g)$ the orthonormal frame bundle $SO(M,g)$ with the Wagner lift $\hat g$ of $g$ is the Riemannian geodesic model in sense of \cite{ShapiroIgoshinYakovlev} for the equation \eqref{eq:2_28}}.

It is worth noting that in our case case $K$ possibly vanishes along a curve $\Sigma$, therefore the metric $\hat g = \bar g$ does not exist along the surface $\hat \Sigma$. The results of \cite{ShapiroIgoshinYakovlev} can be directly applied to the region $M \setminus \Sigma$. For example, Theorem~2 of this paper applied to our case implies that \emph{if $(M,g)$ is geodesically complete, then $(SO(M,g) \setminus \hat \Sigma,\hat g)$ is geodesically complete}.

Note also that, for a compact $M$, the action $\mathcal{L}$ is not multivalued only if $M = \mathbb{T}^2$. Certainly in this case the curvature $K$ has to change sign, so $\Sigma \ne \emptyset$.  

\subsection{Geodesics of the Wagner lift of the metric of a surface of revolution}
\label{subsec:4-3}
In this subsection we consider a surface of revolution $(M,g)$ and find three first integrals for the geodesic equations  of $(SO(M,g),\hat g)$. From this we get two first integrals for the differential equation \eqref{eq:2_28}. 

\subsubsection{First integrals of the system \eqref{eq:2_28} for surfaces of revolution}
Let $(M,g)$ be a surface of revolution, then we can take the coordinate system $(u^1,u^2)$ such that 
\begin{equation*}
g = A^2(u^2) du^1 \otimes du^1 + du^2\otimes du^2
\label{eq:3_300}
\end{equation*}   
Recall that the curvature of $(M,g)$ is
\begin{equation}
K(u^2) = -\frac{A''(u^2)}{A(u^2)}.
\label{eq:3_304}
\end{equation}
Hereafter the prime means derivative with respect to $u^2$. Note also that in this case $\Sigma$ is a disjoint union of parallels. 

We have the orthonormal frame field:
\begin{equation*}
e_1 = \frac{1}{A} \partial_1, \quad e_2 = \partial_2.
\label{eq:3_301}
\end{equation*}
The orthonormal frame field $\{e_1,e_2\}$ induces the coordinates $(u^1,u^2,\varphi)$ on 
the total space $SO(M)$ of the orthonormal frame bundle. 

The group $SO(2)$ acts on $M$ by rotations, we will denote by $r_\alpha$ the rotation at the angle $\alpha$. 
With respect to these coordinates, we have   
\begin{equation*}
r_\alpha(u^1,u^2) = (u^1+\alpha,u^2),
\label{eq:3_305}
\end{equation*}
and, evidently, the complete lift $r_\alpha^c$ of $r_\alpha$, which is an isometry of $(SO(M,g)\setminus \Sigma,\hat g)$ (see Theorem~\ref{prop:5_4}), is written with respect to the coordinates on $SO(M,g)$ as 
\begin{equation*}
r_\alpha^c(u^1,u^2,\varphi) = (u^1+\alpha,u^2,\varphi).
\label{eq:3_306}
\end{equation*}
Now  the structure group $SO(2)$ of the principal bundle $\pi : SO(M,g) \to M$ also acts on $SO(M,g)$ by isometries of $\hat g$, and this action is written with respect to these coordinates as 
\begin{equation*}
R_\beta(u^1,u^2,\varphi) = (u^1,u^2,\varphi+\beta).
\label{eq:3_307}
\end{equation*}
The infinitesimal isometries corresponding to the isometry flows $R_\beta$ and $r_\alpha$, respectively, are written in terms of the orthonormal frame $\mathcal{E}_1, \mathcal{E}_2, \mathcal{E}_3$ as follows  
\begin{equation*}
V_1 = \partial_\varphi = \frac{1}{K(u^2)}\mathcal{E}_3,
\quad 
V_2 = \partial_1 = A(u^2) \mathcal{E}_1 - \frac{A'(u^2)}{K(u^2)} \mathcal{E}_3.
\label{eq:3_308}
\end{equation*} 

It is well known that, \emph{if $W$ is an infinitesimal isometry of a Riemannian manifold $(N,h)$ and $\delta(t)$ is a geodesic of $(N,h)$, then $h_{\delta(t)}(\frac{d}{dt}\delta(t),W(\delta(t)))$ is constant} (this is a generalization of the famous Clairault theorem for geodesics on surfaces of revolution). 

If we apply this statement to $(SO(M,g) \setminus \Sigma,\hat g)$ and a geodesic $\hat\gamma(t)$ of $\hat g$ we get two first integrals of the geodesic equations.
As in Theorem~\ref{thm:5_2_0}, we set
$\frac{d}{dt}\hat\gamma(t) = Q^i(t) \mathcal{E}_i |_{\gamma(t)}$,
then we have that 
\begin{equation}
\hat{g}(V_1,\frac{d}{dt}\hat{\gamma}(t)) =  \frac{1}{K(u^2(t))} Q^3(t) = C_1,
\label{eq:3_310}
\end{equation}
and
\begin{equation}
\hat{g}(V_2,\frac{d}{dt}\hat{\gamma}(t)) = A(u^2(t)) Q^1(t) - \frac{A'(u^2(t))}{K(u^2(t))} Q^3(t) = C_2, 
\label{eq:3_311}
\end{equation}
where $C_1$, $C_2$ are constants.
Note that we already had the first equation (see Theorem~\ref{thm:5_3_0}, (1)).

\begin{theorem}
\label{thm:4_100}
Let $\hat \gamma : (a,b) \to SO(M,g)$ be a generalized geodesic of $(SO(M,g),\hat g)$, and $I_1 = \gamma^{-1}(M \setminus \Sigma)$. 
Let $\gamma = \pi \hat{\gamma}$, and $\frac{d}{dt}\hat\gamma(t) = Q^i(t) \mathcal{E}_i |_{\hat\gamma(t)}$,
Assume that $I_1$ is dense in $I$,
then
\begin{eqnarray}
Q^3(t) = C_1 K(\gamma(t)) = C_1 \frac{A''(\gamma(t))}{A(\gamma(t))}, 
\label{eq:3_315_1}
\\
A(\gamma(t))Q^1(t)-C_1 A'(\gamma(t)) = C_2, 
\label{eq:3_315_2}
\\
(Q^1(t))^2 + (Q^2(t))^2 + C_1^2 K(\gamma(t))^2 = C^2_3,
\label{eq:3_315_3}
\end{eqnarray}
where $C_1$, $C_2$, and $C_3$ are constants. 
\end{theorem} 
\begin{proof}
From \eqref{eq:3_310} we get that $Q^3(t) = C_1 K(\gamma(t))$, then substituting this to \eqref{eq:3_311}, we arrive at \eqref{eq:3_315_2} for $t \in I_1$.  The last equation \eqref{eq:3_315_3} follows from the fact that any geodesic $\hat{\gamma}$ has the tangent vector of constant length, hence the function $(Q^1(t))^2 + (Q^2(t))^2  + (Q^3(t))^2$, $t \in I_1$, is constant. Then, using \eqref{eq:3_315_1}, we get \eqref{eq:3_315_3} for $t \in I_1$.  
Now since  $I_1$ is everywhere dense in $I$, the equations \eqref{eq:3_315_1}--\eqref{eq:3_315_2} hold on $I$. 
\end{proof}

From this theorem we get two first integrals for the differential equation \eqref{eq:2_28}.
\begin{corollary}
\label{cor:3_10}
Let $\gamma : (a,b) \to M$ be a solution of the equation \eqref{eq:2_28}, and $I_1 = \gamma^{-1}(M \setminus \Sigma)$. Let $\frac{d}{dt}\gamma(t) = Q_1(t)e_1 + Q^2(t)e_2$. 
Assume that $I_1$ is dense in $I$, then 
\begin{eqnarray}
A(\gamma(t))Q^1(t)-C A'(\gamma(t)) = C_2, 
\label{eq:3_316_1}
\\
(Q^1(t))^2 + (Q^2(t))^2 + C^2 K(\gamma(t))^2 = C^2_3,
\label{eq:3_316_2}
\end{eqnarray}
where $C$ is the constant of the equation \eqref{eq:2_28}, and $C_2$, $C_3$ are constants. 
\end{corollary}
\begin{proof}
By Theorem~\ref{thm:3_3}, there exists a generalized geodesic $\hat\gamma(t)$, $t \in (a,b)$, such that $I_1 = \gamma^{-1}(M \setminus \Sigma)$ is dense in $(a,b)$. In addition, we have 
\begin{multline*}
\frac{d}{dt}\hat\gamma(t) = Q^1(t) \mathcal{E}_1 |_{\hat\gamma(t)} + Q^2(t) \mathcal{E}_2 |_{\hat\gamma(t)} 
+ Q^3(t) \mathcal{E}_3 |_{\hat\gamma(t)}
\\
\text{ and } 
\frac{d}{dt}\gamma(t) = Q^1(t) e_1 |_{\gamma(t)} + Q^2(t) e_2 |_{\gamma(t)}. 
\label{eq:3_360}
\end{multline*}
From the proof of Theorem~\ref{thm:3_3} it follows that the constant $C$ in the equation \eqref{eq:2_28} is equal to $C_1$. Now \eqref{eq:3_315_2} implies \eqref{eq:3_316_1}, and  \eqref{eq:3_315_3} implies \eqref{eq:3_316_2}.
\end{proof}

\begin{remark}
The equation \eqref{eq:3_316_1} is an analog of the Clairault theorem. 
\end{remark}

\begin{remark}
Corollary~\ref{cor:3_10} can be proved directly.  For example, if we multiply scalarwise equation \eqref{eq:2_28} by $\frac{d\gamma}{dt}$, and then integrate, we get exactly the equality  \eqref{eq:3_316_2} which holds true for all $t$. 
\end{remark}

From Corollary~\ref{cor:3_10} there follow some interesting properties of solutions of \eqref{eq:2_28}. 
From the proof of Theorem~\ref{thm:4_100} it follows that \eqref{eq:3_316_2} holds true without assumption that the metric $g$ admits an infinitesimal isometry. Then from \eqref{eq:3_316_2} immediately follows 

\begin{corollary}
\label{cor:4_5_2}
Let $\gamma : (a,b) \to M$ be a solution of \eqref{eq:2_28} with $C \ne 0$, and $t_0 \in (a,b)$.
Then $\gamma(t)$ lies in the region 
\begin{equation*}
\Omega = \left\{ p \in M \big| |K(p)| \le \sqrt{\frac{1}{C}||\frac{d\gamma}{dt}(t_0)|| + K^2(\gamma(t_0))} \right\}
\label{eq:4_200}
\end{equation*}
If, in addition,  $(M,g)$ is a surface of revolution, then $\Omega$ is a band  consisting of parallels of $M$.
\end{corollary}

\subsubsection{Finding the solutions to \eqref{eq:2_28}}
\begin{figure}[h]
\label{fig:1}
\centerline{\includegraphics[scale=0.3]{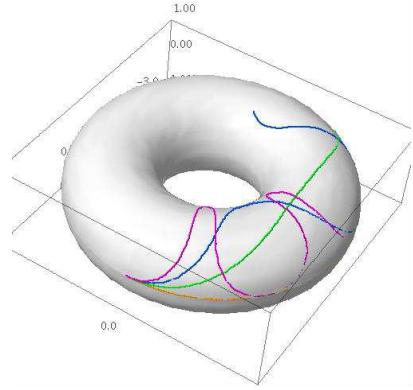}}
\caption{The solutions for \eqref{eq:2_28} with $C = 0, 1, 2, 3$. For $C=0$ we get the geodesic which is the equator of the torus.} 
\end{figure}

\begin{figure}[h]
\label{fig:2}
\centerline{\includegraphics[scale=0.35]{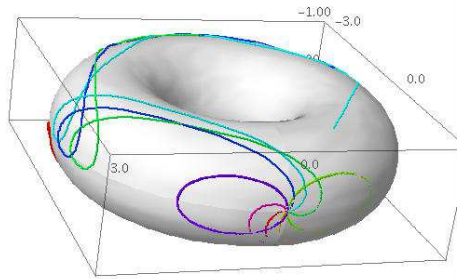}}
\caption{ The solutions for \eqref{eq:2_28} with $C = 1$ passing through a point in different directions.}
\end{figure}

\begin{figure}[h]
\label{fig:3}
\centerline{\includegraphics[scale=0.3]{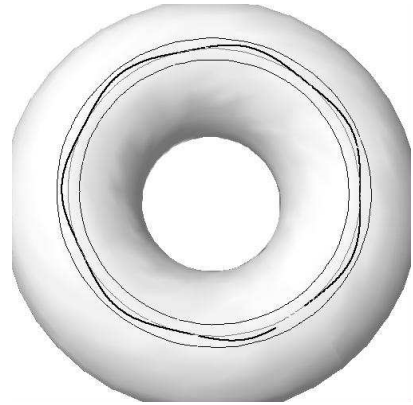}}
\caption{The solution for \eqref{eq:2_28} with $C=3$ which starts at a point of the torus parallel where $K=0$. This solution (the thick curve) cannot go out of the region given by the equation $K < 1/\sqrt{3}$ (bounded by the thin curves).}
\end{figure}
\begin{figure}[h]
\label{fig:4}
\centerline{\includegraphics[scale=0.3]{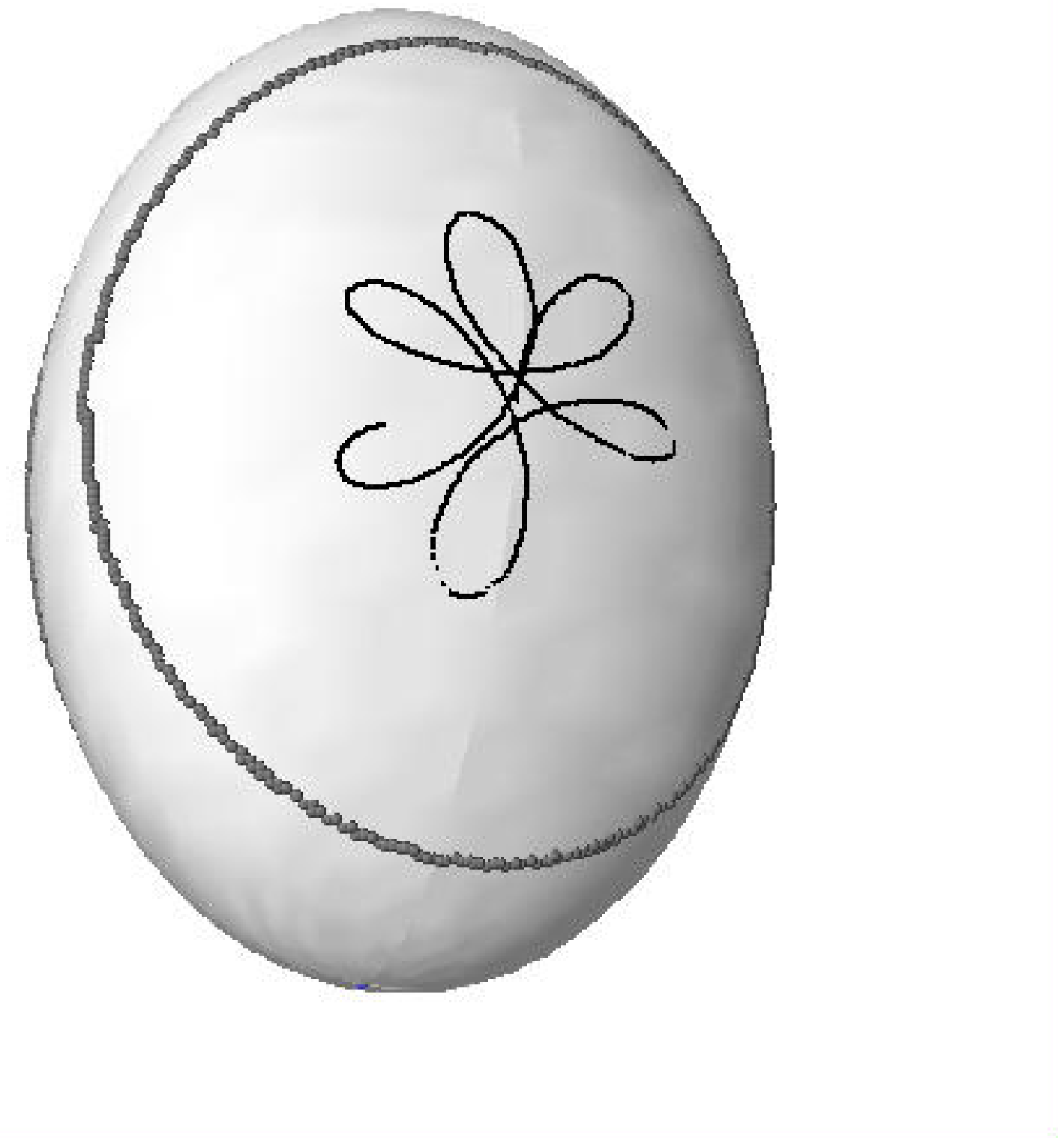}}
\caption{The solution for \eqref{eq:2_28} with $C=20$ on the triaxial ellipsoid (the lengths of axes are 1, 1.5, 2) which starts at a point of minimal curvature of ellipsoid.  This solution cannot go out of the region bounded by the thick gray curve.}
\end{figure}
We know that $Q^1 = ||\frac{d\gamma}{dt}|| \cos\alpha$, where $\alpha$ is the angle between the vector $\frac{d\gamma}{dt}$ and $e_1$, and $||\frac{d\gamma}{dt}|| = \sqrt{(Q^1(t))^2 + (Q^2(t))^2} = \sqrt{C^2_3 - C^2 K(\gamma(t))^2}$. Now from Corollary \ref{cor:3_10} and \eqref{eq:3_304} it follows that  
\begin{equation*}
\cos\alpha = \frac{C_2 + C A'}{A \sqrt{C^2_3 - C^2 K^2}} = \frac{C_2 + C A'}{\sqrt{C^2_3 A^2 - C^2 A''}}. 
\label{eq:3_365}
\end{equation*}
Then 
\begin{equation*}
\cot \alpha = \frac{C_2 + C A'}{\sqrt{C^2_3 A^2 - C^2 A'' - (C_2 + C A')^2}}.
\label{eq:3_370}
\end{equation*}
At the same time,
\begin{equation*}
\cot\alpha = A \frac{du^1}{du^2}.
\label{eq:3_380}
\end{equation*}
Hence follows that, if the solution $\gamma(t)$ is transversal to the parallels $u^1 = \text{const}$, then as a set of points it coincides with the graph of the function
\begin{equation}
u^1(u^2) = \int \frac{C_2 + C A'(u^2)}{A\sqrt{C^2_3 A^2(u^2) - C^2 A''(u^2) - (C_2 + C A'(u^2))^2}} du^2.
\label{eq:3_375}
\end{equation} 
 
However, \eqref{eq:3_375} helps only to find the curves $u^1 = u^1(u^2)$ on which the solutions lie and only in a neighborhood of the initial point. At the same time, for any metric $g$, the solutions of   \eqref{eq:2_28} can be found numerically if we write down the equation \eqref{eq:2_28} with respect to the local coordinates. 
For example, with the use of the computer system SAGE for symbolic and numerical calculations (www.sagemath.org), we can find the solutions of \eqref{eq:2_28} on the torus (see Fig.~1 and 2).    
Also, using numerically found solutions, we can illustrate Corollary~\ref{cor:4_5_2} (see Fig.~3 and~4).

\end{document}